\documentclass[number, sort&compress, preprint,3p,times]{elsarticle}

\usepackage{graphicx}
\usepackage{amssymb}
\usepackage{amsmath}
\usepackage{amsthm}
\usepackage{enumerate}
\usepackage{mdwlist}
\usepackage{longtable}
\usepackage{makecell}
\usepackage{autobreak}
\usepackage{rotating}
\usepackage{caption}
\usepackage[hidelinks,linkcolor=black,bookmarks,bookmarksopen,bookmarksdepth=2]{hyperref}

\makeatletter
\newtheorem{thms}{Theorem}

\newtheorem{remark}{Remark}
\newtheorem{exam}{Example}
\newtheorem{defn}{Definition}
\newtheorem{prop}{Proposition}

\newtheorem{conc}[prop]{Corollary}
\newtheorem{lemma}[prop]{Lemma}

\theoremstyle{remark}

\newcommand{\Z}{\mathbb{Z}}
\newcommand{\PZ}{\mathbb{Z}^{+}}
\newcommand{\F}{\mathbb{F}}
\newcommand{\N}{\mathbb{N}}

\DeclareMathOperator{\PI}{PI}
\DeclareMathOperator{\SP}{SP}

\newcommand{\scell}[2][c]{\begin{tabular}[#1]{@{}c@{}}#2\end{tabular}}

\makeatother

\journal{}

\begin{document}

\begin{frontmatter}



\title{The Zero-Difference Properties of Functions and Their Applications}


\author[GPNU,GZHU1,GZHU2]{Zongxiang~Yi\corref{YZX}}
\cortext[YZX]{Corresponding author at: School of Mathematics and Systems Science, Guangdong Polytechnic Normal University, Guangzhou, 510006, P.R. China.}
\ead{tpu01yzx@gmail.com}

\author[GZHU1,GZHU2]{Dingyi~Pei}
\ead{dypei4188@163.com}

\author[GZHU1,GZHU2]{Chunming~Tang}
\ead{ctang@gzhu.edu.cn}

\address[GPNU]{School of Mathematics and Systems Science, Guangdong Polytechnic Normal University, Guangzhou, 510006, P.R. China.}
\address[GZHU1]{School of Mathematics and Information Science, Guangzhou University, Guangzhou, 510006, P.R. China.}
\address[GZHU2]{Guangdong Provincial Key Laboratory of Information Security, Guangzhou University, Guangzhou, 510006, P.R. China.}

\begin{abstract}
A function $f$ from an Abelian group $(A,+)$ to an Abelian group $(B,+)$ is $(n, m, S)$ zero-difference (ZD), if $S=\{\lambda_\alpha \mid \alpha \in A\setminus\{0\}\}$ where $n=|A|$, $m=|f(A)|$ and $\lambda_\alpha=|\{x \in A \mid f(x+\alpha)=f(x)\}|$. A function is called zero-difference balanced (ZDB) if $S=\{\lambda\}$ where $\lambda$ is a constant number. ZDB functions have many good applications. However it is point out that many known zero-difference balanced functions are already given in the language of partitioned difference family (PDF). The problem that whether zero-difference ``not balanced" functions still have good applications as ZDB functions, is investigated in this paper. By using the change point technic, zero-difference functions with good applications are constructed from known ZDB functions. Then optimal difference systems of sets (DSS) and optimal frequency-hopping sequences (FHS) are obtained with new parameters. Furthermore the sufficient and necessary conditions of these objects being optimal, are given.
\end{abstract}


\end{frontmatter}

\section{Introduction}
Let $(A,+)$ and $(B, +)$ be two finite Abelian groups. Denote $A^*=A\setminus\{0\}$. A function from $A$ to $B$ is an $(n,m,\lambda)$ zero-difference balanced (ZDB) function, if there exists a constant number $\lambda$ such that for any nonzero element $a \in A^*$,
$$|\{x \in A \mid f(x+a)-f(x)=0 \}|=\lambda,$$
where $n=|A|$ and $m=|f(A)|$. \citeauthor{carlet2004highly} first proposed the concept of ZDB function in \citeyear{carlet2004highly}~\cite{carlet2004highly}. Some optimal objects can be obtained by ZDB functions, such as constant composition codes (CCC), constant weight codes (CWC), difference systems of sets (DSS) and frequency-hopping sequences (FHS). Hence many researchers have been working on this topic (see~\cite{carlet2004highly, ding2008optimal, ding2009optimal, zhou2012some, wang2014sets, cai2013new, ding2014three, zha2015cyclotomic, zhifan2015zero, cai2017zero, yi2018generic, yi2019note, liu2019new, xu2024preimage} and the references therein).

A difference family (DF) is collection $\mathcal{F}$ of some nonempty subsets (blocks) $\mathcal{F}_i$ of a group $(G,+)$ such that for any nonzero elements $\alpha$ of $G$, there has exactly $\lambda$ representations of $\alpha$ as a difference of two elements of the same block. If $\mathcal{F}$ is a partition of $G$, then $\mathcal{F}$ is called a  partitioned difference family (PDF). It is well known that a ZDB function is equivalent to a PDF. Many PDFs were constructed since 1970. Thus before a ZDB function is claimed to be new, the authors should check carefully whether the ZDB function is equivalent to an already known PDF. For example, \citeauthor{buratti2019partitioned} recently point out that many recent results on ZDB functions reproduced the earlier results on PDFs~\cite{buratti2019partitioned} and comment that some of these results (even using the terminology of PDFs) are often trivial or even wrong~\cite{buratti2021partitioned}. Since this paper does not concern how to construct PDF, for more details of PDF, the reader is referred to \cite{wilson1972cyclotomy, furino1991difference, buratti20101, fuji2002complete, li2017generic, buratti2019disjoint, buratti2019hadamard, buratti2023partitioned, nakic2023hadamard} and the references therein. In the view of definition, ZDB function has more constraints and PDF is more cleaner. As a result, ZDB function has more useful properties and can be applied to complicated applications. For example, when the underlying group is cyclic, frequency-hopping sequences can be obtained from the blocks of a PDF. But it is impossible to consider the linear complexity of FHSs without equipping an binary operation over blocks. That is why we use the terminology of ZDB function rather than the terminology of PDF when considering further applications.

To construction more optimal objects, we turn to study the generalization of ZDB function, i.e., relaxing the ``balanced" condition on ``zero-difference". There are some work on the generalization of ZDB function. In \citeyear{carlet2014quadratic}, \citeauthor{carlet2014quadratic} proposed an concept called differentially $\delta$-vanishing~\cite{carlet2014quadratic}. A function from $A$ to $B$ is differentially $\delta$-vanishing, if for any nonzero element $a \in A^*$,
$$1\le |\{x \in A \mid f(x+a)-f(x)=0 \}| \le \delta.$$
Any $(n,m,\lambda)$ ZDB function is differentially $\lambda$-vanishing ($\lambda$-DV). But there had been little research on this concept, until \citeauthor{jiang2016generalized} proposed a related concept called generalized zero-difference balanced (G-ZDB) function in \citeyear{jiang2016generalized}~\cite{jiang2016generalized}.
A function from $A$ onto $B$ is an $(n, m, S)$ generalized zero-difference balanced function, if there exists a constant set $S \subset \N$ such that for any nonzero element $a \in A^*$,
$$|\{x \in A \mid f(x+a)-f(x)=0 \}|\in S,$$
where $n=|A|$ and $m=|f(A)|$. Then some objects can be obtained by G-ZDB functions~\cite{jiang2016generalized, Jiang2016New, liu2016some}, but they are not optimal. The main reason is that the size of $S$ is too large, which implies they are not really zero-difference balanced. Therefore, in \citeyear{xu2018optimal}, \citeauthor{xu2018optimal} gave another concept called near zero-difference balanced (N-ZDB) function~\cite{xu2018optimal}.
A function from $A$ to $B$ is an $(n, m, \lambda, t)$ near zero-difference balanced function, if there exist a constant number $\lambda$ and a $t$-subset $T$ of $A^*$ such that for any nonzero element $a \in A^*$,
$$|\{x \in A \mid f(x+a)-f(x)=0 \}| = \begin{cases}
\lambda, & \mbox{if $a \in T$,} \\
\lambda +1, & \mbox{if $a \notin T,$}
\end{cases}$$
where $n=|A|$ and $m=|f(A)|$.

The four concepts ZDB, $\lambda$-DV, G-ZDB and N-ZDB, are close related to each other. Among them, G-ZDB captures most useful information. However the concept of G-ZDB would lead to the misunderstanding that such G-ZDB functions are really ``balanced". Thus the concept of zero-difference (ZD) is used instead of G-ZDB.
\begin{defn}
A function $f$ from an Abelian group $(A,+)$ to an Abelian group $(B,+)$ is $(n, m, S)$ zero-difference (ZD), if $S=\{\lambda_\alpha \mid \alpha \in A\setminus\{0\}\}$ where $n=|A|$, $m=|f(B)|$ and $\lambda_\alpha=|\{x \in A \mid f(x+\alpha)=f(x)\}|$. A function is called an $(n, m, S)$ ZD function if it is $(n, m, S)$ ZD.
\end{defn}

Another concept called ``partition-type difference packing", proposed by~\citeauthor{fuji2004optimal} in~\citeyear{fuji2004optimal}~\cite{fuji2004optimal}, should be introduced for its close connection to ZD functions. A difference packing (DP) is a collection $\mathcal{F}$ of some nonempty subsets (blocks) $\mathcal{F}_i$ of a group $(G,+)$ such that for any nonzero element $\alpha$ of $G$, there are at most $\lambda$ representations of $\alpha$ as a difference of two elements of the same block. If $\mathcal{F}$ is a partition of $G$, then $\mathcal{F}$ is called a partition-type $m$-difference packing (PDP) with parameters $(n, \mathcal{K}, \lambda)$, where $n=|G|$, $m=|\mathcal{F}|$ and $\mathcal{K}=\{|F_i| \mid F_i \in \mathcal{F}\}$. Whereas a ZD function highlights only the size of $\mathcal F$, a difference packing gives full information about the multiset of its block sizes. The partition $\mathcal F$ can be also viewed as the set of preimages of a ZD function. There are cases where, given a ZD function, it is difficult to establish the set of its preimages, i.e., its associated PDF. See~\cite{xu2024preimage} for an interesting paper on this matter. For the differences, ZD function concerns the different numbers of representations for each $\alpha$ while difference packing concerns the maximum number. Hence it is obvious that an $(n,m,S)$ ZD function is equivalent to an $(n, \mathcal{K}, \lambda)$ PDP. For the classical constructions of DPs, the reader is referred to the book~\cite[Section~19.4]{colbourne2007handbook} and the recent paper~\cite{zhang2022existence}. In most of the traditional DPs it is not required that the blocks partition the underlying group G and thus they have the default parameter $\lambda=1$. However, all the ZD functions (equivalently PDP) in this paper have $\lambda > 1$.


Recently \citeauthor{xu2018optimal} and \citeauthor{li2019new} shown that optimal DSSs and FHSs can be obtained via ZD functions in 2018, respectively~\cite{xu2018optimal,li2019new}. Thus we will study the ZD property of cryptographic functions with good applications. Specifically, our contributions are as follows:
\begin{itemize}
\item We propose a framework to construct ZD functions from some type of ZDB functions, namely the change point technic;
\item We obtain optimal CWCs , DSSs and FHSs from ZD functions and generalize the work in~\cite{xu2018optimal};
\item We give the sufficient and necessary conditions of ZD functions having good applications;
\end{itemize}

The rest of this paper is organized as follows. In Section~\ref{se:ogr}, some properties of ZD functions are given and many ZD functions are constructed from some known ZDB functions. In Section~\ref{se:app}, the conditions of the constructed ZD functions having good applications are studied and many optimal objects are obtained. Section~\ref{se:con} conclusions this paper.

\section{A Framework to Construct Zero-Difference Functions}\label{se:ogr}
\subsection{Notations}
Let $(R,+,\times)$ be a ring with an identity. $R^\times$ denote the set of all invertible elements in monoid $(R,\times)$. For a set $A$, define $A^*=A\setminus\{0\}$ if $A$ is a subset of an additive group.

Operations on sets are defined as usual. For any two set $A, B$, $A + B=\{a + b\mid a\in A, b\in B\}$ if $a + b$ is well defined. For any element $x$ and any set $A$, $x+A=\{x+a\mid a\in A\}$, $xA=\{xa \mid a \in A\}$. Similarly, $A-B$, $x-A$, $Ax$, $A-x$ and $A+x$ can be defined in this manner.

To resolve the ambiguity, several notions are defined. The size of a set $S$ is denoted by $|S|$. Let $f$ be a function from $A$ to $B$. The set of all images is denoted by $f(A)=\{f(x)\mid x \in A\}$. The preimage of an element $b\in B$ is denoted by $f^{-1}(b)=\{x\in A \mid f(x)=b\}$. The set of all preimages is denoted by $\PI(f)=\{f^{-1}(b) \mid b \in B\}$. The multi-set of sizes of all preimage is denoted by $\SP(f)=\{|S| \mid S\in \PI(f)\}$. Using the language of combinatorics, a multi-set can be represented as $\{r_1^{e_1},r_2^{e_2},\ldots, r_k^{e_k}\}$ which means that $r_i$ appears exactly $e_i$ times. If $e_i=1$, then $e_i$ can be omitted.

\subsection{Some Properties of Zero-Difference functions}
In this subsection, let $f$ be an $(n, m, S)$ ZD function from $(A, +)$ onto $(B, +)$. Define $r_b=|f^{-1}(b)|$, for every $b \in B$. Define $\lambda_\alpha=|\{x \in A \mid f(x+\alpha)=f(x)\}|$, for every nonzero element $\alpha \in A^*$. Denote $\lambda=\max_{i\in S}{i}$ and $\mu=\min_{i\in S}{i}$.

The following lemma follows directly from the definition of ZD property.
\begin{lemma}\label{lm:property_A_of_ZD}
Notations are as above. Then
$$\sum_{b\in B}{r_b (r_b - 1)}= \sum_{\alpha \in A^*}{\lambda_\alpha},$$
\end{lemma}
\begin{proof}
Consider the following set
$$\{ (x, \alpha)\in A\times A^* \mid f(x+\alpha)-f(x)=0 \}.$$
On one hand, given $\alpha \in A^*$, the number of $x$ satisfying the equation is $\lambda_\alpha$. It equals the right hand side when $\alpha$ runs over $A^*$. On the other hand, given $b\in B$, the number of $x$ satisfying $f(x)=b$ is $r_b$ and the number of $\alpha$ satisfying the equation is $r_b-1$ for each $x$. It leads to the left hand side when $b$ runs over $B$.
\end{proof}

Let $\overline{\lambda}$ be the arithmetic average of the multi-set $\{\lambda_\alpha \mid \alpha \in A^*\}$, i.e.,
\begin{equation}\label{eq:define_of_avg_lambda}
\overline{\lambda}=\frac{1}{n-1}\sum_{\alpha \in A^*}{\lambda_\alpha}.
\end{equation}
Based on Lemma~\ref{lm:property_A_of_ZD}, we have the following lower bound on $\lambda$.
\begin{lemma}\label{lm:property_B_of_ZD}
Notations are as above. Then
\begin{equation}\label{eq:inequality_of_lambda}
\lambda\ge \left\lceil \frac{(n-\epsilon)(n+\epsilon-m)}{m(n-1)} + \lambda-\overline{\lambda} \right\rceil,
\end{equation}
where $n=km+\epsilon$ with $0\le \epsilon<m$. In particular,
$$\lambda=\frac{(n-\epsilon)(n+\epsilon-m)}{m(n-1)} + \lambda-\overline{\lambda},$$
if and only if, for $b\in B$, $r_b=k$ for $m-\epsilon$ times and $r_b=k+1$ for the other $\epsilon$ times.
\end{lemma}
\begin{proof}
According to Lemma~\ref{lm:property_A_of_ZD}, we have
$$\sum_{b\in B}{r_b^2}-n= \sum_{\alpha \in A^*}{\lambda_\alpha}=(n-1)\overline{\lambda}=(n-1)(\lambda-\lambda+\overline{\lambda}).$$
Moreover, we have
$$\sum_{b\in B}{r_b^2}-n\ge \min_{r_b,b\in B}\sum_{b\in B}{r_b^2}-n$$
Note that $\sum_{b\in B}{r_b}=n$. By integral programming, $\{r_b\mid b \in B\}$ attains the minimum value, if and only if, $f$ is as balanced as possible. Note that $n=km+\epsilon$. We have
$$\min_{r_b,b\in B}\sum_{b\in B}{r_b^2}-n=(m-\epsilon)k^2+\epsilon (k+1)^2-n=\frac{(n-\epsilon)(n+\epsilon-m)}{m}.$$
It completes the proof by some simple calculations.
\end{proof}

Using Lemma~\ref{lm:property_A_of_ZD}, we can also obtain the bounds on the size of preimage sets. The sizes of all preimage sets are important in some applications, such as construct constant composition codes and difference systems of sets.
\begin{lemma}\label{lm:property_C_of_ZD}
Notations are as above. Then for each $b\in B$,
$$\frac{n-\sqrt{\Delta}}{m}\le r_b \le \frac{n+\sqrt{\Delta}}{m},$$
where $\Delta=(n+\lambda n-\lambda)m^2-(n^2+n+\mu n-\mu)m+n^2$.
\end{lemma}
\begin{proof}
For any $b \in B$, we have
\begin{equation*}\begin{split}
0&\le\sum_{b_1,b_2\in B\setminus\{b\},b_1\ne b_2}{\left(r_{b_1}-r_{b_2}\right)^2} \\
&=\sum_{b_1,b_2\in B\setminus\{b\},b_1\ne b_2}{r_{b_1}^2+r_{b_2}^2-2r_{b_1}r_{b_2}} \\
&=2(m-2)\sum_{b_0\in B\setminus\{b\}}{r_{b_0}^2}-2\sum_{b_1,b_2\in B\setminus\{b\},b_1\ne b_2}{r_{b_1}r_{b_2}}.
\end{split}\end{equation*}
It then follows that
$$(m-2)\sum_{b_0\in B\setminus\{b\}}{r_{b_0}^2} \ge \sum_{b_1,b_2\in B\setminus\{b\},b_1\ne b_2}{r_{b_1}r_{b_2}}.$$
By Lemma~\ref{lm:property_A_of_ZD}, we have
\begin{equation*}\begin{split}
n + (n-1)\mu & \le \sum_{b_0\in B}{r_{b_0}^2}   \\
&=\left(\sum_{b_0\in B\setminus\{b\}}{r_{b_0}^2}\right)+r_b^2 \\
&=\left(\sum_{b_0\in B\setminus\{b\}}{r_{b_0}^2}\right)+\left(n-\sum_{b_0\in B\setminus\{b\}}{r_{b_0}}\right)^2 \\
&=2\left(\sum_{b_0\in B\setminus\{b\}}{r_{b_0}^2}\right)+n^2-2n(n-r_b)+\sum_{b_1,b_2\in B\setminus\{b\},b_1\ne b_2}{r_{b_1}r_{b_2}} \\
&\le m\left(\sum_{b_0\in B}{r_{b_0}^2}\right)+mr_b^2+n^2-2n(n-r_b) \\
&\le m(n+(n-1)\lambda)+mr_b^2+n^2-2n(n-r_b).
\end{split}\end{equation*}
Finally, solving the above inequality of $r_b$ completes the proof.
\end{proof}

In a similar way, it also has the following lemmas.
\begin{lemma}\label{lm:property_AA_of_ZD}
Notations are as above. Then
$$\sum_{b\in B}{r_b^2}=(n-1)\overline{\lambda}+n.$$
\end{lemma}

\begin{lemma}\label{lm:property_BB_of_ZD}
Notations are as above. Then
$$\overline{\lambda}\ge \frac{(n-\epsilon)(n+\epsilon-m)}{m(n-1)},$$
where $n=km+\epsilon$ with $0\le \epsilon<m$. In particular,
$$\overline{\lambda}=\frac{(n-\epsilon)(n+\epsilon-m)}{m(n-1)},$$
if and only if, for $b\in B$, $r_b=k$ for $m-\epsilon$ times and $r_b=k+1$ for the other $\epsilon$ times.
\end{lemma}

\begin{lemma}\label{lm:property_CC_of_ZD}
Notations are as above. Then for each $b\in B$,
$$\frac{n-\sqrt{\Delta}}{m}\le r_b \le \frac{n+\sqrt{\Delta}}{m},$$
where $\Delta=(n+\overline{\lambda} n-\overline{\lambda})m^2-(n^2+n+\overline{\lambda} n-\overline{\lambda})m+n^2$.
\end{lemma}

\begin{remark}
As a comparison, the properties of ZD functions in Lemmas~\ref{lm:property_A_of_ZD},~\ref{lm:property_B_of_ZD} and~\ref{lm:property_C_of_ZD}, are generalizations of those properties of ZDB functions in~\cite{wang2014sets} and  generalization of those properties of N-ZDB functions in~\cite{xu2018optimal}. 
\end{remark}

\subsection{Zero-Difference Functions from Zero-Difference Balanced Functions}
Firstly, a generic class of ZDB function proposed by~\citeauthor{yi2018zero} is recalled.
\begin{prop}\cite[Theorem 1]{yi2018generic}\label{prop:construct_zdb_on_generic_ring}
Let $(R, +, \times)$ be a ring of order $n$, and let $G$ be a subgroup of $(R, \times)$. Denote $D_G=\{rG\mid r \in R\}$, where $rG=\{rg\mid g\in G\}$. Define a function from $R$ to $\Z_{|D_G|}$, $f_G(x)=h_G(g_G(x))$ where $g_G(x)=rG$ if $x\in rG$ and $h_G$ is a bijection from $D_G$ to $\Z_{|D_G|}$. If $G$ satisfies the condition
\begin{equation}\label{eq:condition_of_zdb_on_generic_ring}
(G-1)\setminus\{0\} \subset R^\times,
\end{equation}
then $f_G(x)$ is an $(n, \frac{n-1}{k}+1, k-1)$ ZDB function from $(R,+)$ to $(\Z_{m},+)$, where $m=\frac{n-1}{k}+1$ and $k=|G|$.
\end{prop}
\begin{remark}
Let $R$ be a residual class ring $\Z_n$ or a product of finite fields $\F_q$, then these ZDB functions are studied in~\cite{cai2013new, zha2015cyclotomic, ding2014three} and also in \cite{buratti2019disjoint, xiang2020new} in terms of PDF. In fact, all the ZDB functions (PDFs) can be obtained using nearring theory~\cite{clay1992nearrings} which was published in~\citeyear{clay1992nearrings}.
\end{remark}
\begin{remark}
Some readers, specially those from PDF community, may be confused that why we choose this ZDB function. There are two reasons. On the one hand, the preimages of this function can be easily changed to be almost balanced and balanced function usually has good properties. On the other hand, it is obvious that a minor change to a function can only lead to a minor change to the zero-difference property. As a result, some good applications from a ZDB function are likely to be reserved when a minor change was made. We will discuss the ``likely" in detail in Section~\ref{se:app}.
\end{remark}
So far as we know, there only two subclasses of such ZDB functions, namely Proposition~\ref{prop:zdb_1_on_zn} and Proposition~\ref{prop:zdb_1_on_fq}. The main difference between them is that ZDB functions in Proposition~\ref{prop:zdb_1_on_zn} are cyclic and thus have more applications and ZDB functions in Proposition~\ref{prop:zdb_1_on_fq} have more feasible parameters.
\begin{prop}\cite[Corollary 3.5]{buratti2019disjoint}\cite[Theorem 1]{zha2015cyclotomic}\label{prop:zdb_1_on_zn}
Let $n=p_1^{r_1}p_2^{r_2}\cdots p_k^{r_k}$, where $p_1<p_2<\cdots < p_k$ are odd prime numbers, and $r_1, r_2, \ldots, r_k$ are positive integers. Then for any positive integers $e$ such that $e\mid \gcd(p_1-1, p_2-1, \ldots, p_k-1)$, there exists an $(n,\frac{n-1}{e}+1,e-1)$ ZDB function from $(\Z_{n},+)$ to $(\Z_{\frac{n-1}{e}+1},+)$.
\end{prop}

\begin{prop}\cite[Corollary 3.3]{buratti2019disjoint}\cite[Theorem 1]{ding2014three}\label{prop:zdb_1_on_fq}
Let $n=p_1^{r_1}p_2^{r_2}\cdots p_k^{r_k}$, where $p_1<p_2<\cdots < p_k$ are prime numbers, and $r_1, r_2, \ldots, r_k$ are positive integers. Then for any positive integers $e$ such that $e\mid \gcd(p_1^{r_1}-1, p_2^{r_2}-1, \ldots, p_k^{r_k}-1)$, there exists an $(n,\frac{n-1}{e}+1,e-1)$ ZDB function from $(\prod_{i=1}^{k}\F_{p_i^{r_i}},+)$ to $(\Z_{\frac{n-1}{e}+1},+)$.
\end{prop}

If the ZDB function $f_G$ is modified a little bit, then a ZD function $f_G^0$ can be obtained. Comparing with the original ZDB functions, these ZD functions are better in balance, namely they are almost balanced.
\begin{thms}\label{th:construct_nzdb_on_generic_ring}
Using the notations in Proposition~\ref{prop:construct_zdb_on_generic_ring}, define a function
$$f_G^0(x)=\begin{cases}
f_G(x), & x \ne 0, \\
f_G(1), & x = 0.
\end{cases}$$
Then $f_G^0$ is an $(n, \frac{n-1}{k}, S)$ ZD function, where
$$S=\begin{cases}
\{n\}, & \mbox{if $\frac{n-1}{k}=1$}, \\
\{k\}, & \mbox{if $\frac{n-1}{k}=2$ and $-1 \notin G$ }, \\
\{k-1,k\}, & \mbox{if $\frac{n-1}{k}>2$ and $-1 \notin G$ }, \\
\{k-1,k+1\}, & \mbox{if $\frac{n-1}{k}\ge 2$ and $-1 \in G$ }.
\end{cases}$$
\end{thms}
\begin{proof}
Obviously $|f_G^0(R)|=|f_G(R)|-1=\frac{n-1}{k}$, since there does not exist $x \in R$ such that $f_G^0(x)=f_G(0)$. If $\frac{n-1}{k} =1$, then $f_G^0$ must be a constant function. In the following, assume that $\frac{n-1}{k}\ge2$. To solve the equation $f_G^0(x+a)=f_G^0(x)$, consider two special cases.
\begin{itemize}
\item Case $a \in G$: We assert that $x=0$ is a solution of the equation $f_G^0(x+a)=f_G^0(x)$. Since $$f_G^0(0+a)=f_G^0(a)=f_G(a)=h_G(G),$$
    and
    $$f_G^0(0)=f_G(1)=h_G(G),$$
    we have $f_G^0(0+a)=f_G^0(0)$, where $h_G$ is the bijection in Proposition~\ref{prop:construct_zdb_on_generic_ring}.
\item Case $a \in -G$: We assert that $x=-a$ is a solution of the equation $f_G^0(x+a)=f_G^0(x)$. Since $$f_G^0(-a+a)=f_G^0(0)=f_G(1)=h_G(G),$$
    and
    $$f_G^0(-a)=f_G(-a)=h_G(G),$$
    we have $f_G^0(-a+a)=f_G^0(-a)$.
\end{itemize}
Note that neither $x=0$ nor $x=-a$ can be a solution of the equation $f_G(x +a)=f_G(x)$ and that $-G=G$ if and only if $-1\in G$. Now for any nonzero element $a \in R^*$, we have the following different cases.
\begin{itemize}
\item If $\frac{n-1}{k}\ge 2$ and $-1 \in G$, then we have
\begin{equation*}\begin{split}
&\{ x \in R \mid f_G^0(x +a)=f_G^0(x) \} \\
=&\begin{cases}
\{ x \in R^* \mid f_G(x +a)=f_G(x) \} \cup \{ 0,-a\}, & a\in G,\\
\{ x \in R^* \mid f_G(x +a)=f_G(x) \} , & otherwise.\\
\end{cases}
\end{split}\end{equation*}
Therefore, we have
\begin{equation*}\begin{split}
&|\{ x \in R \mid f_G^0(x +a)=f_G^0(x) \}|=\begin{cases}
k+1, & a\in G,\\
k-1 , & otherwise.\\
\end{cases}
\end{split}\end{equation*}

\item If $\frac{n-1}{k}>2$ and $-1 \notin G$, then we have
\begin{equation*}\begin{split}
&\{ x \in R \mid f_G^0(x +a)=f_G^0(x) \} \\
=&\begin{cases}
\{ x \in R^* \mid f_G(x +a)=f_G(x) \} \cup \{ 0\}, & a\in G,\\
\{ x \in R^* \mid f_G(x +a)=f_G(x) \} \cup \{ -a\}, & a\in -G,\\
\{ x \in R^* \mid f_G(x +a)=f_G(x) \} , & otherwise,\\
\end{cases}
\end{split}\end{equation*}
and
\begin{equation*}\begin{split}
&|\{ x \in R \mid f_G^0(x +a)=f_G^0(x) \}|=\begin{cases}
k, & a\in -G\cup G,\\
k-1 , & otherwise.\\
\end{cases}
\end{split}\end{equation*}
\item If $\frac{n-1}{k}=2$ and $-1 \notin G$, we have
\begin{equation*}\begin{split}
&\{ x \in R \mid f_G^0(x +a)=f_G^0(x) \} \\
=&\begin{cases}
\{ x \in R^* \mid f_G(x +a)=f_G(x) \} \cup \{ 0\}, & a\in G,\\
\{ x \in R^* \mid f_G(x +a)=f_G(x) \} \cup \{ -a\}, & a\in -G,\\
\end{cases}
\end{split}\end{equation*}
and
\begin{equation*}\begin{split}
|\{ x \in R \mid f_G^0(x +a)=f_G^0(x) \}|=k, a\in -G\cup G.
\end{split}\end{equation*}
\end{itemize}
It completes the proof.
\end{proof}
To investigate when $-1\in G$, we have the following lemma.
\begin{lemma}\label{lm:about_p2}
Let $(R,+,\times)$ be a ring and $G$ be a subgroup of $(R,\times)$. Denote $k=|G|$. If $G$ satisfies $(G-1)\setminus\{0\}\subset R^\times$, then $-1 \in G$ if and only if $2\mid k$ or the characteristic $p$ of $R$ is $2$.
\end{lemma}
\begin{proof}
If $2 \mid k$, then by Cauchy's Theorem \cite{bhattacharya1994basic} there exists an element $\alpha \in G$ with order $2$, i.e., $\alpha^2=1$. Then
$$(\alpha-1)^2=\alpha^2-2\alpha+1=2(1-\alpha).$$
Since $(G-1)\setminus\{0\}\subset R^\times$, we have $\alpha-1\in R^\times$. Hence
$$(\alpha-1)=-2.$$
So $-1=\alpha\in G$.

If $p=2$, then $-1=1\in G$.

Conversely, suppose that $-1 \in G$. If $p\ne 2$, then $-1\ne 1$. The multiplicative order of $-1$ must be $2$ since $(-1)^2=1$. By Lagrange's Theorem~\cite{bhattacharya1994basic} we have $2\mid k$.
\end{proof}

\begin{remark}
Note that the results in \cite[Section 4]{xu2018optimal} are also the main results in~\cite{cai2013new} and Theorem~\ref{th:construct_nzdb_on_generic_ring} is a generalization of them. Hence $f_G^0$ is a generalization of the exact N-ZDB function in \cite[Theorem 4.1]{xu2018optimal}. The technic used in Theorem~\ref{th:construct_nzdb_on_generic_ring} is called change point technic.
\end{remark}

Using the same technic, we can generalize Theorem~\ref{th:construct_nzdb_on_generic_ring} a bit.
\begin{thms}\label{th:construct_nzdb_on_generic_group}
Let $f$ be an $(n,\frac{n-1}{k},\lambda)$ ZDB function from $(A,+)$ to $(B,+)$ such that for any $b \in B$,
\begin{equation*}\begin{split}
&|\{ x \in A \mid f(x)=b \}|=\begin{cases}
k, & b \ne b_0,\\
1, & b = b_0 ,
\end{cases}
\end{split}\end{equation*}
where $k\ge 2$ and $b_0\in B$. Suppose $f(a_0)=b_0$. For any $a \in A$, define $I(a)=\{x \in A \mid f(x)=f(a)\}$ and define a function
$$g_a(x)=\begin{cases}
f(x), & x \ne a_0, \\
f(a), & x = a_0.
\end{cases}$$
If $a_0\notin I(a)$ then $g_a$ is an $(n, \frac{n-1}{k}, S)$ ZD function where
$$S=\begin{cases}
\{n\}, & \mbox{if $\frac{n-1}{k}=1$}, \\
\{k\}, & \mbox{if $\frac{n-1}{k}=2$ and $D=\emptyset$ }, \\
\{k-1,k\}, & \mbox{if $\frac{n-1}{k}>2$ and $D=\emptyset$ }, \\
\{k-1, k+1\}, & \mbox{if $\frac{n-1}{k}\ge 2$ and $D=I(a)-a_0$},\\
\{k-1, k, k+1\}, & \mbox{if $\frac{n-1}{k}\ge 2$ and $\emptyset \varsubsetneq D\varsubsetneq I(a)-a_0$},
\end{cases}$$
and $D=(I(a)-a_0) \cap (a_0-I(a))$. Moreover, $\lambda=k-1$.
\end{thms}
\begin{proof}
The proof is similar with that of Theorem~\ref{th:construct_nzdb_on_generic_ring}. For any nonzero element $\alpha\in A^*$, if $\alpha$ belongs to either $I(a)-a_0$ or $a_0-I(a)$, then it will append one more solution for $g_a(x +\alpha)=g_a(x)$, either $x=a_0$ or  $x=a_0-\alpha$. If $\alpha$ belongs to both $I(a)-a_0$ and $a_0-I(a)$, then it will append two more solutions for $g_a(x +\alpha)=g_a(x)$, $x=a_0$ and $x=a_0-\alpha$. For the rest, the solutions for $g_a(x +\alpha)=g_a(x)$ remain the same.

Define $\lambda_\alpha=|\{x\in A \mid g_a(x+\alpha)=g_a(x)\}|$ for every nonzero element $\alpha \in A^*$. Denote $E=(I(a)-a_0) \cup (a_0-I(a))$. Assume that $\frac{n-1}{k}\ge 2$. Then we have the following four cases.

\begin{itemize}
\item If $\frac{n-1}{k}=2$ and $D=\emptyset$, then $E=A^*$, $a_0=0$ and we have
\begin{equation*}\begin{split}
&\lambda_\alpha=\lambda+1,  \alpha\in E.
\end{split}\end{equation*}

\item If $\frac{n-1}{k}>2$ and $D=\emptyset$, then we have
\begin{equation*}\begin{split}
&\lambda_\alpha=\begin{cases}
\lambda+1, & \mbox{if $\alpha\in I(a)-a_0$},\\
\lambda+1, & \mbox{if $\alpha\in a_0-I(a)$},\\
\lambda , & \mbox{otherwise}.\\
\end{cases}
\end{split}\end{equation*}

\item If $\frac{n-1}{k}\ge 2$ and $D=I(a)-a_0$, then $D=I(a)-a_0 = a_0-I(a)$ and we have
\begin{equation*}\begin{split}
&\lambda_\alpha=\begin{cases}
\lambda+2, & \mbox{if $\alpha\in D$},\\
\lambda , & \mbox{otherwise}.\\
\end{cases}
\end{split}\end{equation*}

\item If $\frac{n-1}{k}\ge 2$ and $\emptyset \varsubsetneq D\varsubsetneq I(a)-a_0$, then  we have
\begin{equation*}\begin{split}
&\lambda_\alpha=\begin{cases}
\lambda+1, & \mbox{if $\alpha\in E\setminus D$},\\
\lambda+2, & \mbox{if $\alpha\in D$},\\
\lambda , & \mbox{otherwise}.\\
\end{cases}
\end{split}\end{equation*}
\end{itemize}

To finish the proof, we assert that $\lambda=k-1$. On one hand, since $g_a$ satisfies the condition in Lemma~\ref{lm:property_BB_of_ZD}, we have
$$\overline{\lambda}=\frac{(n-\epsilon)(n+\epsilon-m)}{m(n-1)}=k-1+\frac{2k}{n-1},$$
where $m=\frac{n-1}{k}$ and $\epsilon=1$. On the other hand, note that $|I(a)-a_0|=|a_0-I(a)|=k$ and $|E|=2k-|D|$. Hence according to the definition of $\overline{\lambda}$, i.e., Equation \eqref{eq:define_of_avg_lambda}, we have
\begin{equation*}\begin{split}
\overline{\lambda}=&\frac{|D|(\lambda+2)+(|E|-|D|)(\lambda+1)+(n-1-|E|)\lambda}{n-1} \\
=&\lambda+\frac{2k}{n-1}.
\end{split}\end{equation*}
Consequently $\lambda=k-1$.
\end{proof}
\begin{remark}
Let $a_0=0$ and $a=1$. Then the ZDB functions in Proposition~\ref{prop:construct_zdb_on_generic_ring} satisfy the conditions in Theorem~\ref{th:construct_nzdb_on_generic_group}. A survey of such ZDB functions is in~\cite{yi2019note}.
\end{remark}

\begin{defn}
Let $f$ be a function $A$ to $B$. $f$ is Type-A if $\SP(f)=\{1,e^{m-1}\}$, and is Type-B if $\SP(f)=\{e+1,e^{m-1}\}$ where $m=|f(A)|$ and $e$ is a positive integer.
\end{defn}

\begin{prop}\cite[Corollary 1]{yi2019note}\label{prop:zdb_2_on_zn}
Let $n=p_1^{r_1}p_2^{r_2}\cdots p_k^{r_k}$, where $p_1<p_2<\cdots < p_k$ are odd prime numbers, and $r_1, r_2, \ldots, r_k$ are positive integers. Then for any positive integers $e$ such that $e(e-1)\mid \gcd(p_1-1, p_2-1, \ldots, p_k-1)$, there exist $(en,\frac{en-1}{e-1}+1,e-2)$ ZDB functions from $(\Z_{en},+)$ to $(\Z_{\frac{en-1}{e-1}+1},+)$.
\end{prop}

\begin{prop}\cite[Corollary 2]{yi2019note}\label{prop:zdb_2_on_fq}
Let $n=p_1^{r_1}p_2^{r_2}\cdots p_k^{r_k}$, where $p_1<p_2<\cdots < p_k$ are prime numbers, and $r_1, r_2, \ldots, r_k$ are positive integers. Denote $R=\prod_{i=1}^{k}\F_{p_i^{r_i}}$. Then for any positive integer $e$ such that $e(e-1)\mid \gcd(p_1^{r_1}-1, p_2^{r_2}-1, \ldots, p_k^{r_k}-1)$, there exist $(en,\frac{en-1}{e-1}+1,e-2)$ ZDB functions from $(R\times \Z_e,+)$ to $(\Z_{\frac{en-1}{e-1}+1}, +)$.
\end{prop}

It is interesting that Theorem~\ref{th:construct_nzdb_on_generic_group} gives many ZD functions from ZDB functions without caring how these ZDB functions are constructed. In particular, it converts a ZDB function of Type-A into a ZD function of Type-B. In~\cite{yi2019note}, many ZDB functions of Type-A are summarized, namely Proposition~\ref{prop:zdb_1_on_zn}, Proposition~\ref{prop:zdb_1_on_fq}, Proposition~\ref{prop:zdb_2_on_zn} and Proposition~\ref{prop:zdb_2_on_fq}. Although not all the generated ZD functions have good applications, we will show what kind of ZD functions would lead to optimal objects, such as constant weight codes, difference systems of sets and frequency-hopping sequences.

\section{Applications of Zero-Difference Functioins}\label{se:app}
We remark that the frameworks of applications using ZDB functions are still valid if ZDB functions are replaced by ZD functions. Thus many objects can be obtained by ZD functions. In this section, we will give the conditions of these objects being optimal. In this paper, we only concern whose optimal objects, since optimal objects can save more electronic power.


\subsection{Constant Composition Codes}
An $(n, M, d, [w_0, w_1, \ldots, w_{q-1}])_q$ constant composition code (CCC) is a code over an Abelian group $G$ with length $n$, size $M$ and minimum Hanmming distance $d$ such that in every codeword the element $b_i$ appears exactly $w_i$ times for every $i$, where $b_i\in G$.
Let $A_q(n, d, [w_0, w_1, \ldots, w_{q-1}])$ be the maximum size of an $(n, M, d, [w_0, w_1, \ldots, w_{q-1}])_q$ CCC. A CCC is optimal if the bound in Lemma~\ref{lm:bound_of_CCC} is met.
\begin{lemma}~\cite{luo2003constant}\label{lm:bound_of_CCC}
If $nd-n^2+\sum_{i=0}^{q-1}w_i^2>0$, then
$$A_q(n, d, [w_0, w_1, \ldots, w_{q-1}])\le\frac{nd}{nd-n^2+\sum_{i=0}^{q-1}w_i^2}.$$
\end{lemma}
Obviously ZD functions can be used to construct CCCs and these CCCs in Proposition~\ref{prop:construct_ccc} is optimal only if the ZD function is a ZDB function.
\begin{prop}\cite{ding2008optimal}\label{prop:construct_ccc}
Denote
$$A=\{a_0, a_1, \ldots, a_{n-1}\}, B=\{b_0, b_1, \ldots, b_{m-1}\}.$$
Let $f$ be a function from $A$ to $B$. If $f$ is an $(n, m, S)$ ZD function, then
\begin{equation}\label{eq:define_of_code}
\mathcal{C}_f=\{(f(a_0+a_i), \ldots, f(a_{n-1}+a_i))\mid 0\le i \le n-1\}
\end{equation}
is an $(n, n, n-\lambda, [w_0, w_1, \ldots, w_{m-1}])_m$ CCC over $B$, where $w_i=|\{x\in A \mid f(x)=b_i\}|$ and $\lambda = \max_{x \in S}{x}$.
\end{prop}

\begin{prop}\label{th:app_of_ccc}
In Proposition~\ref{prop:construct_ccc}, if $\mathcal{C}_f$ is optimal, then $S=\{ \lambda\}$, i.e., $f$ is a ZDB function.
\end{prop}
\begin{proof}
If $\mathcal{C}_f$ is optimal, then
$$n=\frac{nd}{nd-n^2+\sum_{i=0}^{q-1}{w_i^2}}, $$
where $d=n-\lambda$ and $q=m$. We have
$$\sum_{i=0}^{m-1}{w_i^2}=n\lambda+n-\lambda.$$
It follows from Lemma~\ref{lm:property_A_of_ZD} that
$$n\lambda-\lambda=\sum_{\alpha \in A^*}{\lambda_\alpha},$$
where $\lambda_\alpha=|\{x \in A \mid f(x+\alpha)=f(x)\}|$ for every $\alpha \in A^*$.
It implies
$$n-1=\sum_{\alpha \in A^*}{\frac{\lambda_\alpha}{\lambda}}.$$
We have $0\le \frac{\lambda_\alpha}{\lambda} \le 1$ , since $\lambda_\alpha \in S$ and $\lambda =\max_{x \in S}{x}$. Therefore, $\frac{\lambda_\alpha}{\lambda}=1$, for every $\alpha \in A^*$. It implies that $S=\{\lambda\}$, i.e., $f$ is a ZDB function.
\end{proof}

\subsection{Constant Weight Codes}

An $(n, M, d, w)_q$ constant weight code (CWC) is a code over an Abelian group $\{b_0, b_1, \ldots, b_{q-1}\}$ with length $n$, size $M$ and minimum Hamming distance $d$ such that the Hamming weight of each codeword is $w$. Let $A_q(n, d, w)$ be the maximum size of an $(n, M, d, w)_q$ CWC.

Note that CCCs are special CWCs. They have many applications such as determining the zero error decision feedback capacity of discrete memoryless channels~\cite{telatar1989zero}, multiple-access communications~\cite{dyachkov1984random}, spherical codes for modulation~\cite{ericson1995spherical}, DNA codes~\cite{king2003bounds, milenkovic2005design}, powerline communications~\cite{chu2004constructions, colbourn2004permutation} and frequency hopping~\cite{chu2006constant}.

A CWC is optimal if the bound in Lemma~\ref{lm:bound_of_CWC} is met. Theorem~\ref{th:app_gzdb_cwc} gives many optimal CWCs from ZD functions.

\begin{lemma}~\cite{fu1998constructions}\label{lm:bound_of_CWC}
If $nd-2nw+\frac{q}{q-1}w^2>0$, then
$$A_q(n, d, w)\le\frac{nd}{nd-2nw+\frac{q}{q-1}w^2}.$$
\end{lemma}

\begin{thms}\label{th:app_gzdb_cwc}
Let $f$ be an $(n,m,S)$ ZD function. Then $\mathcal{C}_f$ in \eqref{eq:define_of_code} is an $(n, n, n-\lambda, n-b_0)_m$ CWC, where $\lambda = \max_{x \in S}{x}$ and $b_0=|f^{-1}(0)|$. Furthermore, $\mathcal{C}_f$ is optimal if and only if $\lambda(n-1)(m-1)=b_0^2m-2b_0n+n(n-m+1)$.
\end{thms}

We have several results for the ZDB functions of Type-A in Proposition~\ref{prop:construct_zdb_on_generic_ring} and the ZD functions of Type-B in Theorem~\ref{th:construct_nzdb_on_generic_group}. Comparing with Theorem 6 in~\cite{yi2018generic}, Theorem~\ref{th:app_cwc1} is generic since it does not depend on the construction method.
\begin{thms}\label{th:app_cwc1}
Let $f$ be an $(n,\frac{n-1}{k}+1,k-1)$ ZDB function of Type-A such that $|f^{-1}(0)|=1$. Then $\mathcal{C}_f$ in \eqref{eq:define_of_code} is an optimal $(n, n, n-k+1, n-1)_{\frac{n-1}{k}+1}$ CWC.
\end{thms}

\begin{thms}\label{th:app_cwc2}
Let $f$ be an $(n,2,k)$ ZDB function of Type-B in Theorem~\ref{th:construct_nzdb_on_generic_group}. Define $b_0=|f^{-1}(0)|$. Then $\mathcal{C}_f$ in \eqref{eq:define_of_code} is an optimal $(n, n, n-k, n-b_0)_{2}$ CWC where $b_0=k$ or $b_0=k+1$.
\end{thms}

Note that the construction of code in Proposition~\ref{prop:construct_ccc} does not require the function should be defined over a cyclic group. Since Proposition~\ref{prop:zdb_1_on_fq} and Proposition~\ref{prop:zdb_2_on_fq} give more parameters that Proposition~\ref{prop:zdb_1_on_zn} and Proposition~\ref{prop:zdb_2_on_zn}. Only Proposition~\ref{prop:zdb_1_on_fq} and Proposition~\ref{prop:zdb_2_on_fq} are considered when applying Theorem~\ref{th:app_cwc1} and Theorem~\ref{th:app_cwc2}. Consequently, we obtain many optimal CWCs. Some optimal CWCs are listed in Table~\ref{tb:app_cwc_new}. For more optimal CWCs, please see~\cite{brouwer2019cwc}.

\begin{conc}\label{conc:app_cwc1}
Let $n=p_1^{r_1}p_2^{r_2}\cdots p_k^{r_k}$, where $p_1<p_2<\cdots < p_k$ are odd prime numbers, and $r_1, r_2, \ldots, r_k$ are positive integers. Then for any positive integers $e$ such that $e\mid \gcd(p_1^{r_1}-1, p_2^{r_2}-1, \ldots, p_k^{r_k}-1)$, there exists an optimal $(n, n, n-e+1,n-1)_{\frac{n-1}{e}+1}$ CWC.
\end{conc}
\begin{exam}
Let $n=11$ and $e=5$. Then an optimal $(11 , 11 , 7 , 10)_3$ CWC $\mathcal{C}_1$ is obtained from an $(11, 3, 4)$ ZDB function by Corollary~\ref{conc:app_cwc1}, namely, \begin{align*}\begin{autobreak}\mathcal{C}_1=\{
12221212101,
21222120111,
22210111122,
22102211211,
12021121212,
21121112220,
12112122021,
20111222112,
11122201122,
01211221221,
11212012212
\}.\end{autobreak}\end{align*}
\end{exam}

\begin{conc}\label{conc:app_cwc2}
Let $n=p_1^{r_1}p_2^{r_2}\cdots p_k^{r_k}$, where $p_1<p_2<\cdots < p_k$ are prime numbers, and $r_1, r_2, \ldots, r_k$ are positive integers. Denote $R=\prod_{i=1}^{k}\F_{p_i^{r_i}}$. Then for any positive integer $e$ such that $e(e-1)\mid \gcd(p_1^{r_1}-1, p_2^{r_2}-1, \ldots, p_k^{r_k}-1)$, there exists an optimal $(en, en, en-e+2,en-1)_{\frac{en-1}{e-1}+1}$ CWC.
\end{conc}
\begin{exam}
Let $n=7$ and $e=3$. Then an optimal $(21, 21, 20, 20)_{11}$ CWC $\mathcal{C}_2$ is obtained from an $(21, 11, 1)$ ZDB function by Corollary~\ref{conc:app_cwc2}, namely, \begin{align*}\begin{autobreak}\mathcal{C}_2=\{
2304467A876319A158952,
3443078759A26895611A2,
042235A89193856A76714,
43232A91158456779A860,
40324859A8727116A5693,
675A89801449736232A51,
78A95814530192A342766,
A7819047386545291A326,
8591A15362468274039A7,
7915843825760A316942A,
6A9874064711382A95235,
323429156610AA8987574,
16857794803AA45223619,
985613252A8A494671037,
A96716A27328548430195,
15A7623941A9264857308,
5679A34106982735A8241,
816A522A3957310786449,
91786A73942560132458A,
5A169562A237139044878,
224031667A54975819A83
\}.\end{autobreak}\end{align*}
\end{exam}

\begin{conc}\label{conc:app_cwc3}
Let $n=p_1^{r_1}p_2^{r_2}\cdots p_k^{r_k}$, where $p_1<p_2<\cdots < p_k$ are odd prime numbers, and $r_1, r_2, \ldots, r_k$ are positive integers. Then for any odd and positive integers $e$ such that $e\mid \gcd(p_1^{r_1}-1, p_2^{r_2}-1, \ldots, p_k^{r_k}-1)$, there exists an optimal $(n, n, \frac{n+1}{2},\frac{n-1}{2})_{2}$ CWC and an optimal $(n, n, \frac{n+1}{2},\frac{n+1}{2})_{2}$ CWC.
\end{conc}
\begin{exam}
Let $n=11$ and $e=5$. Then an optimal $(11 , 11 , 6 , 5)_{2}$ CWC $\mathcal{C}_3$  and an optimal $(11 , 11 , 6 , 6)_{2}$ CWC $\mathcal{C}_4$ are obtained from an $(11, 2, 5)$ ZDB function by Corollary~\ref{conc:app_cwc3}, namely, \begin{align*}\begin{autobreak}\mathcal{C}_3=\{
01110101000,
10111010000,
11100000011,
11001100100,
01010010101,
10010001110,
01001011010,
10000111001,
00011100011,
00100110110,
00101001101
\}, \end{autobreak} \\ \begin{autobreak} \mathcal{C}_4=\{
01110101010,
10111011000,
11101000011,
11011100100,
01110010101,
10010001111,
01001011110,
11000111001,
00011110011,
10100110110,
00101101101
\}.\end{autobreak}\end{align*}
\end{exam}


\begin{table}[!htbp]
\centering
\caption{\text{Some Optimal $(n,n,d,w)_2$ Constant Weight Codes}}\label{tb:app_cwc_new}
\begin{tabular}{|c|c|c|c||c|c|c|c|}
\hline
$n$ & $d$ & $w$ & $A_2(n, d, w)$ & $n$ & $d$ & $w$ & $A_2(n, d, w)$ \\
\hline
5 & 4 & 4 & 5 & 5 & 4 & 5 & 5 \\
7 & 4 & 4 & 7 & 7 & 4 & 5 & 7 \\
7 & 6 & 6 & 7 & 7 & 6 & 7 & 7 \\
9 & 8 & 8 & 9 & 9 & 8 & 9 & 9 \\
11 & 6 & 6 & 11 & 11 & 6 & 7 & 11 \\
11 & 10 & 10 & 11 & 11 & 10 & 11 & 11 \\
13 & 10 & 10 & 13 & 13 & 10 & 11 & 13 \\
13 & 12 & 12 & 13 & 13 & 12 & 13 & 13 \\
15 & 14 & 14 & 15 & 15 & 14 & 15 & 15 \\
17 & 16 & 16 & 17 & 17 & 16 & 17 & 17 \\
19 & 10 & 10 & 19 & 19 & 10 & 11 & 19 \\
19 & 16 & 16 & 19 & 19 & 16 & 17 & 19 \\
19 & 18 & 18 & 19 & 19 & 18 & 19 & 19 \\
23 & 12 & 12 & 23 & 23 & 12 & 13 & 23 \\
27 & 14 & 14 & 27 & 27 & 14 & 15 & 27 \\
31 & 16 & 16 & 31 & 31 & 16 & 17 & 31 \\
\hline
\end{tabular}
\end{table}

\subsection{Difference Systems of Sets}
Difference systems of sets (DSS) are related with comma-free codes, authentication codes and secrete sharing schemes~\cite{Ogata2004New,fuji2009perfect}. Let $\{D_0, D_1, \ldots, D_{q-1}\} $ be disjoint subsets of an Abelian group $(G, +)$. Denote $|G|=n$ and $|D_i|=w_i$ for every $i$. Then $\{D_0, D_1, \ldots, D_{q-1}\} $ is said to be an $(n,\{w_0, w_1, \ldots, w_{q-1}\}, \lambda)$ DSS if the multi-set
$$\{x-y \mid \begin{array}{c}
x\in D_i, y\in D_j, 0\le i\ne j \le q-1
\end{array}
\}$$
contains every nonzero element $g \in G$ at least $\lambda$ times. Moreover, a DSS is perfect if every non-zero element $g$ appears exactly $\lambda$ times in the multi-set just mentioned above. It is required that
$$\tau_q(n,\lambda)=\sum_{i=0}^{q-1}\left|D_i\right|$$
as small as possible. A DSS is called optimal if the bound in Lemma~\ref{lm:bound_of_DSS} is met.
\begin{lemma}~\cite{wang2006new}\label{lm:bound_of_DSS}
For an $(n,[w_0, w_1, \ldots, w_{q-1}], \lambda)$ DSS, we have
$$\begin{array}{c}
\tau_q(n,\lambda) \ge \sqrt{SQUARE(\lambda(n-1) + \left\lceil \frac{\lambda(n-1)}{q-1}\right\rceil)}
\end{array},$$
where $SQUARE(x)$ denotes the smallest square number that is no less than $x$ and $\lceil x \rceil$ denotes the smallest integer that no less that $x$.
\end{lemma}
\begin{remark}
DSSs on non-cyclic groups are related to authentication codes and secret sharing schemes~\cite{Ogata2004New, fuji2009perfect}.
\end{remark}
In \citeyear{ding2009optimal}, \citeauthor{ding2009optimal} gave a method to construct optimal DSSs from a ZDB function.
\begin{prop}\cite{ding2009optimal}\label{prop:construct_dss}
If $f$ is an $(n, m, \lambda)$ ZDB function from $A$ to $B$. Denote $f(A)=\{b_0, b_1, \ldots, b_{m-1}\}$. Then
\begin{equation}\label{eq:define_of_dss}
\mathcal{D}=\{ D_i \mid 0 \le i \le q - 1\}
\end{equation}
is an $(n,[w_0, w_1, \ldots, w_{m-1}], n - \lambda)$ DSS if $n\ge m\lambda$, where $D_i=\{x \in A \mid f(x)=b_i\}$, $w_i=|D_i|$.
\end{prop}

To obtain optimal DSSs, we need the following lemmas to prove Theorem~\ref{th:app_gzdb_dss}.
\begin{lemma}\label{lm:about_ceiling_A}
Let $a$ be a positive integer, and let $b$ be a real number. $\lceil x \rceil$ denotes the ceiling function. Then $a<\lceil b\rceil$, if and only if, $a < b$.
\end{lemma}
\begin{proof}
For some $0\le\varepsilon < 1$, we have
\begin{equation*}\begin{split}
a < \lceil b\rceil & \Leftrightarrow  a \le \lceil b\rceil - 1 \Leftrightarrow  a \le b + \varepsilon -1 \\
& \Leftrightarrow  a \le b - (1 - \varepsilon)   \Leftrightarrow  a < b.
\end{split}\end{equation*}
\end{proof}

\begin{lemma}\label{lm:about_ceiling_B}
Let $a$ be a positive integer, and let $b$ be a real number. $\lceil x \rceil$ denotes the ceiling function. Then $a \ge \lceil b\rceil$, if and only if, $a \ge b$.
\end{lemma}

\begin{thms}\label{th:app_gzdb_dss}
In Proposition~\ref{prop:construct_dss}, the DSS $\mathcal{D}$ is optimal, if and only if $n \ge m\lambda-m+2$.
\end{thms}
\begin{proof}
Note that $m\ge 2$. If $n \ge m\lambda-m+2$, then $n-1 > m(\lambda -1)$. That is $\frac{n-\lambda}{m-1} > \lambda -1$. Hence, we have
\begin{equation*}\begin{split}
& \rho(n-1) + \left\lceil \frac{\rho(n-1)}{q-1}\right\rceil \\
\ge& (n-\lambda)(n-1)+\frac{(n-\lambda)(n-1)}{m-1} \\
=& (n-1)\left(n-\lambda+\frac{(n-\lambda)(n-1)}{m-1}\right) \\
>& (n-1)(n-\lambda+\lambda-1) \\
=&(n-1)^2,
\end{split}\end{equation*}
where $q=m$ and $\rho=n-\lambda$. Since $\tau_q(n,\rho)=n$, it follows from Lemma~\ref{lm:bound_of_DSS} that
$$\sqrt{SQUARE\left(\rho(n-1) + \left\lceil \frac{\rho(n-1)}{q-1}\right\rceil\right)}=n.$$
Therefore, $D$ meets the bound in Lemma~\ref{lm:bound_of_DSS}.

Conversely, if $D$ is optimal, then we have
$$\sqrt{SQUARE\left((n-\lambda)(n-1) + \left\lceil \frac{(n-\lambda)(n-1)}{m-1}\right\rceil\right)}=n.$$
That is
$$(n-\lambda)(n-1) + \left\lceil \frac{(n-\lambda)(n-1)}{m-1}\right\rceil > (n-1)^2.$$

According to Lemma~\ref{lm:about_ceiling_A}, we can get
$$(n-\lambda)(n-1) + \frac{(n-\lambda)(n-1)}{m-1} > (n-1)^2,$$
which leads to $n-1 > m(\lambda -1)$. Finally we have $n \ge m\lambda-m+2$.
\end{proof}
\begin{remark}
Since N-ZDB function is a special case of ZD function, Theorem~\ref{th:app_gzdb_dss} is a generalization of Theorem 5.12 in~\cite{xu2018optimal}.
\end{remark}
\begin{remark}
When the ZD function $f$ is also a ZDB function, the condition of the constructed DSSs being optimal in this paper, is weaker than that in Proposition~\ref{prop:construct_dss}~\cite{ding2009optimal} (see also \cite[Lemma 6]{zhou2012some}). If $m \ge 2$, then $n \ge m\lambda-m+2$ which implies $n\ge m \lambda$.
\end{remark}
As a result, it improves Theorem 7 in~\cite{yi2018generic} a bit. Moreover, Theorem~\ref{th:improve_yzx_dss} does not depend on the construction method and will give more optimal DSSs.
\begin{lemma}\cite[Theorem 7]{yi2018generic}\label{lm:dss_optimal_old_condition}
Let $f$ be an $(n, \frac{n-1}{k}+1,k-1)$ ZDB function of Theorem 3 in~\cite{yi2018generic}. Then the DSS constructed by the method in Proposition~\ref{prop:construct_dss} is optimal if $n\ge (k-1)^2$.
\end{lemma}
\begin{thms}\label{th:improve_yzx_dss}
Let $f$ be an $(n, \frac{n-1}{k}+1,k-1)$ ZDB function. Then the DSS constructed by the method in Proposition~\ref{prop:construct_dss} is optimal if $n\ge \frac{k(k-1)}{2}+1$.
\end{thms}
\begin{exam}
Using the notations in Proposition~\ref{prop:construct_zdb_on_generic_ring}, put $R=\Z_{11}$ and $G=\langle 4 \rangle$. Then the group $G$ of order $5$ satisfies Condition \eqref{eq:condition_of_zdb_on_generic_ring}. Hence there is an $(11, 3, 4)$ ZDB function $f$ which lead to an $(11, [1,5,5], 7)$ DSS. Obviously this DSS is optimal. It is easy to check that $f$ satisfies the condition in Theorem~\ref{th:improve_yzx_dss}, but not the condition in Lemma~\ref{lm:dss_optimal_old_condition}.
\end{exam}

No matter it is Type-A or Type-B, a ZD function may be used to construct an optimal DSS. Several classes of optimal DSSs are obtained by ZD functions in Theorem~\ref{th:construct_nzdb_on_generic_group}.
\begin{thms}\label{th:app_dss1}
Let $f$ be an $(n,\frac{n-1}{k},S)$ ZD function in Theorem~\ref{th:construct_nzdb_on_generic_group} such that $k=\max_{x \in S}{x}$. Then the DSS $\mathcal{D}$ in Proposition~\ref{prop:construct_dss} is optimal. Moreover, it is perfect if $f$ is a ZDB function.
\end{thms}

\begin{conc}\cite[Theorem 5.13]{xu2018optimal}\label{con:app_dss1}
Let $n$ be an positive integer with the following factorization $n=\prod_{i=1}^{r}{p_i^{e_i}}$,
where $p_i$ are odd prime numbers, $e_i$ are positive integers $(i=1,2,\ldots,r)$. Let $e\ge 2$ be an odd integer such that $e \mid p_i-1$ for every $i$. Then there exists an optimal $(n, [\underbrace{e,e,\ldots,e}_{\frac{n-1}{e}-1\text{ times}},e+1], n-e)$ DSS over $(\mathbb{Z}_{n},+)$.
\end{conc}
\begin{conc}\label{con:app_dss2}
Let $n$ be an positive integer with the following factorization $n=\prod_{i=1}^{r}{p_i^{e_i}}$,
where $p_i$ are prime numbers, $e_i$ are positive integers $(i=1,2,\ldots,r)$. Let $e\ge 2$ be an odd integer such that $e \mid p_i^{e_i}-1$ for every $i$. Then there exists an optimal $(n, [\underbrace{e,e,\ldots,e}_{\frac{n-1}{e}-1\text{ times}},e+1], n-e)$ DSS over $(\prod_{i=1}^{r}\F_{p_i^{e_i}},+)$.
\end{conc}

\begin{conc}\label{con:app_dss3}
Let $n$ be an positive integer with the following factorization $n=\prod_{i=1}^{r}{p_i^{e_i}}$,
where $p_i$ are odd prime numbers, $e_i$ are positive integers $(i=1,2,\ldots,r)$. Let $e\ge 2$ be an odd integer such that $e(e-1) \mid p_i-1$ for every $i$. Then there exists an optimal $(en, [\underbrace{e-1,e-1,\ldots,e-1}_{\frac{en-1}{e-1}-1\text{ times}},e], en-e+1)$ DSS over $(\Z_{en},+)$.
\end{conc}

\begin{conc}\label{con:app_dss4}
Let $n$ be an positive integer with the following factorization $n=\prod_{i=1}^{r}{p_i^{e_i}}$,
where $p_i$ are prime numbers, $e_i$ are positive integers $(i=1,2,\ldots,r)$. Let $e\ge 2$ be an odd integer such that $e(e-1) \mid p_i^{e_i}-1$ for every $i$. Then there exists an optimal $(en, [\underbrace{e-1,e-1,\ldots,e-1}_{\frac{en-1}{e-1}-1\text{ times}},e], en-e+1)$ DSSs over $(\Z_{e}\times \prod_{i=1}^{r}{\F_{p_i^{e_i}}},+)$.
\end{conc}

\begin{remark}
Corollary~\ref{con:app_dss2} is a generalizations of Theorem 5.13 in~\cite{xu2018optimal} and provide more optimal DSSs with new parameters. In the following, Example~\ref{ex:app_dss1} can not be obtained by Theorem 5.13 in~\cite{xu2018optimal}.
\end{remark}
\begin{exam}\label{ex:app_dss1}
In Theorem~\ref{th:app_dss1}, put $n=3^2\times 5=45$ and $e=4$. Then we obtain an optimal $(45, [\underbrace{4,4,\ldots,4}_{10\text{ times}},5], 41)$ DSS.
\end{exam}

We summarize some known optimal DSSs in Table~\ref{tb:known_optimal_dss} where $p$ and $q$ denote a prime number and a prime power, respectively.
\begin{sidewaystable}[!tbp]
\centering
\caption{\text{Some known optimal DSSs with parameters $(n, [ w_0, w_1, \ldots, w_{M-1} ], \rho)$}}\label{tb:known_optimal_dss}
\begin{tabular*}{\linewidth }{c|c|c|c|c}
\hline
$n$ & $w_0, w_1, \ldots, w_{M-1}$ & $\rho$ & Constraints & Reference.   \\
\hline
$q^2+1$ & $w_0, w_1, \ldots, w_{q-1}$ & $q^2-q$ & $m\in \PZ$ and $q=2^m$ &  \cite[Proposition 10]{ding2008optimal}  \\
$q$ & $w_0, w_1, \ldots, w_{d-1}$ & $ q-\frac{q-d}{d} $ & $d\mid q $ & \cite[Proposition 7]{ding2008optimal} \\
$p^2$ & $2p-1,p-1,\ldots,p-1$ & $p^2-p $ & $-$ & \cite[Corollary 8]{ding2009optimal}\\
$\frac{q^m-1}{N}$ & $w_0, w_1, \ldots, w_{q-1}$ & $\frac{q^{m-1}(q-1)}{N}$ & $N\mid q-1$ and $\gcd(N,m)=1$ & \cite[Theorem 9]{ding2009optimal}\\
$q^m$ & $w_0, w_1, \ldots, w_{\frac{q^m-1}{d}}$ & $q^m-d+1$ & $d\mid q-1$ & \cite[Proposition 14]{ding2012zero} \\
$q^2+q+1$ & $w_0, w_1, \ldots, w_{q-1}$ & $q^2+2$ & $-$ & \cite[Proposition 22]{ding2012zero} \\
$q^m-1$ & $w_0, w_1, \ldots, w_{q^s-1}$ & $q^m-q^{m-s}$ & $1\le s \le m $ & \cite[Corollary 1]{zhou2012some} \\
$t\frac{q^m-1}{N}$ & $w_0, w_1, \ldots, w_{q^s-1}$ & $t\frac{q^m-q^{m-s}}{N}$ & \scell{$N\mid q-1$, $\gcd(N,m)=1$, \\ $1\le t \le N$ and $1\le s \le m $ } & \cite[Theorem 6]{zhou2012some} \\
$n$ & $1,e,e,\ldots,e$ & $n-e+1 $ & \scell{ $n=\prod_{i=1}^{k}{p_i^{e_i}}>3$, $p_i$ is odd prime \\ and $e\mid (p_i-1)$ for all $i$ } & \cite[Theorem 1]{cai2013new} \\
$ev$ & $1,e-1,e-1,\ldots,e-1$ & $ev-e+2 $ & \scell{ $v=\prod_{i=1}^{k}{p_i^{e_i}}$, $p_i$ is odd prime \\ and $e(e-1)\mid (p_i-1)$ for all $i$ } & \cite[Theorem 1]{cai2017zero} \\
$n$ & $1,e,e,\ldots,e$ & $n-e+1 $ & \scell{ $n=\prod_{i=1}^{k}{p_i^{e_i}}$, $n\ge (m-1)^2$ \\ and $e\mid (p_i^{e_i}-1)$ for all $i$ } & \cite[Theorem 7]{yi2018generic} \\
$n$ & $e+1,e,e,\ldots,e$ & $n-e $ & \scell{ $n=\prod_{i=1}^{k}{p_i^{e_i}}$, $p_i$ is odd prime, \\ $e$ is odd and $e\mid (p_i-1)$ for all $i$ } & \cite[Theorem 5.13]{xu2018optimal} \\
$ev$ & \scell{$1,e-1,e-1,\ldots,e-1$,\\ $e-2,e-2,\ldots,e-2$} & $ev-e+2 $ & \scell{ $v=\prod_{i=1}^{k}{p_i^{e_i}}$, $p_i$ is odd prime,$e\ge 3$, \\ $(e-2)\mid (p_i-1)$ and $e\mid (p_i-1)$ for all $i$ } & \cite[Theorem 5.14]{xu2018optimal} \\
$ev$ & \scell{$e-1,e-1,\ldots,e-1$,\\ $e-2,e-2,\ldots,e-2$} & $ev-e+2 $ & \scell{ $v=\prod_{i=1}^{k}{p_i^{e_i}}$, $p_i$ is odd prime, \\ $e$ is odd with $e\ge 3$, \\ $(e-2)\mid (p_i-1)$ and $e\mid (p_i-1)$ for all $i$ } & \cite[Theorem 5.15]{xu2018optimal} \\
$n$ & $e+1,e,e,\ldots,e$ & $n-e $ & \scell{ $n=\prod_{i=1}^{k}{p_i^{e_i}}$, $p_i$ is odd prime, \\ $e$ is odd and $e\mid (p_i^{e_i}-1)$ for all $i$ } & Corollary~\ref{con:app_dss1} \\
$ev$ & $e,e-1,e-1,\ldots,e-1$ & $ev-e+1 $ & \scell{ $v=\prod_{i=1}^{k}{p_i^{e_i}}$, $p_i$ is odd prime, \\ $e$ is odd and $e(e-1)\mid (p_i-1)$ for all $i$ } & Corollary~\ref{con:app_dss2} $^*$ \\
$ev$ & $e,e-1,e-1,\ldots,e-1$ & $ev-e+1 $ & \scell{ $v=\prod_{i=1}^{k}{p_i^{e_i}}$, $p_i$ is odd prime, \\ $e$ is odd  and $e(e-1)\mid (p_i^{e_i}-1)$ for all $i$ } & Corollary~\ref{con:app_dss3} $^*$ \\
\hline
\end{tabular*}
\end{sidewaystable}

\subsection{Frequency-hopping Sequences}
For any two sequences $X, Y$ of length $n$ over an alphabet $B$. The Hamming correlation $H_{X,Y}$ is defined as
$$H_{X,Y}(t)=\sum_{i=0}^{n-1}{h[x_i,y_{(i+t)\pmod{n}}]},0\le t < n$$
where $h[a,b]=1$ if $a=b$, and 0 otherwise. Frequency-hopping sequences (FHS) are the sequences such that the hamming autocorrelation is as small as possible.
\citeauthor{lempel1974families} gave a lower bound in \citeyear{lempel1974families}~\cite{lempel1974families}. Let $(n,m,\lambda)$ denote an FHS $X$ of length $n$ over an alphabet of size $m$ with $\lambda=H(X)$. An FHS is optimal if the bound in Lemma~\ref{lm:bound_of_one_FH} is met.
\begin{lemma}~\cite{lempel1974families}\label{lm:bound_of_one_FH}
For any FHS $X$ of length $n$ over an alphabet of size $m$, define
$$H(X)=\max_{1\le t< n}\{H_{X,X}(t)\},$$
then
\begin{equation}\label{eq:bound_of_one_FH}
H(X) \ge  \left\lceil \frac{(n-\epsilon)(n+\epsilon-m)}{m(n-1)} \right\rceil,
\end{equation}
where $\epsilon$ is the least nonnegative residue of $n$ modulo $m$ and $\lceil x \rceil$ denotes the smallest integer that no less that $x$.
\end{lemma}
In \citeyear{wang2014sets}, \citeauthor{wang2014sets} gave a method to construct FHSs by a ZDB function.
\begin{prop}\cite[Lemma 5.3]{wang2014sets}\label{prop:construct_fhs}
Let $f$ be an $(n,m,S)$ ZDB function from a cyclic group $(A,+)$ to a group $(B,+)$. Then $\mathcal{T}=\{f(i \alpha)\}_{i=0}^{n-1}$ is an $(n, m, \lambda)$ FHS, where $\alpha$ is a generator of $A$ and $\lambda =\max_{x \in S}{x}$.
\end{prop}
\begin{remark}
The criteria of optimality in this subsection is different from that in~\cite{liu2016some}. In this subsection, we call an FHS is optimal with respect to the bound in Lemma~\ref{lm:bound_of_one_FH}. In~\cite{liu2016some}, they call an FHS is optimal with respect to a set of FHSs (see Lemma 12 in~\cite{liu2016some}).
\end{remark}

Theorem 5.5, Theorem 5.6 and Theorem 5.7 in~\cite{xu2018optimal} imply that, a ZD function $f$ is not necessary ZDB to obtain optimal FHS. In order to characterize such ZD functions, we will give the conditions of the constructed FHSs being optimal is given as follows.
\begin{thms}\label{th:app_fhs1}
In proposition~\ref{prop:construct_fhs}, the FHS $\mathcal{T}$ is optimal, if and only if, $0\le \lambda-C<1$, where $C=\frac{(n-\epsilon)(n+\epsilon-m)}{m(n-1)}$ and $n=km+\epsilon$ with $0\le \epsilon < m$.
\end{thms}
\begin{proof}
If the FHS $T$ is optimal, then we have
$$\lambda=\lceil C \rceil.$$
Thus
$$\lambda < \lceil C \rceil + \delta.$$
where $0< \delta = C+1-\lceil C \rceil \le 1$.
Note that $\lceil C \rceil=C+1-\delta$. Hence
$$\lambda < C+1.$$
It follows from Lemma~\ref{lm:property_B_of_ZD} that $\lambda \ge C$. Therefore, $0\le \lambda-C<1$.

Conversely, if $0\le \lambda-C<1$, then we have
$$\lambda<C+1.$$
Denote $0<\delta = C+1-\lceil C \rceil \le 1$.
Note that $C=\lceil C \rceil -1+\delta.$ Hence
$$\lambda<C+1=\lceil C \rceil -1+\delta+1=\lceil C \rceil+\delta.$$
Thus
$$\lambda \le \lceil C \rceil.$$
It follows from Lemma~\ref{lm:bound_of_one_FH} that $\lambda \ge \lceil C \rceil$. Therefore, $\lambda=\lceil C \rceil$, i.e., the FHS $T$ is optimal.
\end{proof}

\begin{remark}
Since N-ZDB function is a special case of ZD function, Theorem~\ref{th:app_fhs1} is a generalization of Lemma 5.4, Theorem 5.5, Theorem 5.6 and Theorem 5.7 in~\cite{xu2018optimal}.
\end{remark}

Note that $0\le \lambda-C= \delta_1 + \delta_2<1$, where $\delta_1=\lambda-\overline{\lambda}$ and $\delta_2=\overline{\lambda}-C$. For an $(n,m,S)$ ZD function with $\lambda-C<1$, it implies that the deviation $\delta_1$ of the zero-difference distribution $S$ should be not too great, and that the distance $\delta_2$ between the preimage distribution $\overline{\lambda}=\frac{(\sum_{b\in B}{r_b^2})-n}{n-1}$ of function $f$ and the idea balanced preimage distribution $C=\frac{\min_{r_b,b\in B}\sum_{b\in B}{r_b^2}-n}{n-1}$ should be not too great. Furthermore, the upper bound of the sum of $\delta_1 + \delta_2$ is 1. Therefore, $\overline{\lambda}$ is a bridge connecting the zero-difference property and the preimage distribution.

According to Lemma~\ref{lm:property_BB_of_ZD}, $C=\overline{\lambda}$ if $f$ is almost balanced. Thus we only consider a special class of ZD functions which are almost balanced, namely, Type-B. For a ZD function of Type-B, $\lambda -C< 1$ is equivalent to $\lambda - \overline{\lambda} < 1$.

\begin{defn}\label{def:AB}
A function $f$ from $A$ onto $B$ is almost balanced (AB) if for $b\in B$, $w_b=k$ for $m-\epsilon$ times and $w_b=k+1$ for the other $\epsilon$ times, where $w_b=|\{x \in A \mid f(x)=b\}|$, $n=|A|$, $m=|B|$ and $n=km+\epsilon$ with $0\le \epsilon < m$.
\end{defn}

\begin{thms}\label{th:app_fhs2}
Let $f$ be an $(n,m,S)$ ZD and AB function. Then the FHS $\mathcal{T}$ in Proposition~\ref{prop:construct_fhs} is optimal if and only if $\lambda - \overline{\lambda} < 1$ where $\lambda =\max_{x \in S}{x}$ and $\overline{\lambda}$ is defined in~\eqref{eq:define_of_avg_lambda}.
\end{thms}

%
%

\begin{thms}\label{th:app_fhs3}
Let $f$ be an $(n, \frac{n-1}{k}, S)$ ZD functions in Theorem~\ref{th:construct_nzdb_on_generic_group} such that $k=\max_{x \in S}{x}$. Then the FHS $\mathcal{T}$ in Proposition~\ref{prop:construct_fhs} is optimal.
\end{thms}
\begin{proof}
Obviously, $f$ is Type-B and thus AB. From the proof of Theorem~\ref{th:construct_nzdb_on_generic_group}, we have
$$\overline{\lambda}=k-1+\frac{2k}{n-1}.$$
As $\frac{n-1}{k}\ge 2$, it is easy to see that $0< \frac{2k}{n-1}\le 1$ and $0\le 1-\frac{2k}{n-1} < 1$. It follows that
$$0\le \lambda-\overline{\lambda}=k-(k-1+\frac{2k}{n-1})=1-\frac{2k}{n-1}< 1.$$
It completes the proof by Theorem~\ref{th:app_fhs2}.
\end{proof}
\begin{remark}
From the conclusion of Theorem~\ref{th:construct_nzdb_on_generic_group}, if $k=\max_{x \in S}{x}$, then it must be $D=\emptyset$. According to Lemma~\ref{lm:about_p2}, it requires that $k$ must be odd if the characteristic of the underlying ring is not $2$.
\end{remark}

Using the above theorems, it is easy to obtain optimal FHSs from ZD functions.
\begin{sidewaystable}[!tbp]
\centering
\caption{\text{Some known optimal FHSs with parameters $(n,m,\lambda)$}}\label{tb:known_optimal_fhs}
\begin{tabular*}{\linewidth }{c|c|c|c|c}
\hline
$n$ & $m$ & $\lambda$ & Constraints & Reference.   \\
\hline
$p$ & $e$ & $f$ & $p=ef+1$ is prime, $e$ is even and $f$ is odd  & \cite[Corollary 2]{chu2005optimal} \\
$p$ & $e+1$ & $f-1$ & $p=ef+1$ is prime and $2\le f \le e+2$  & \cite[Corollary 3]{chu2005optimal} \\
$p$ & $L$ & $2g$ & $p=2Lg+1$ is an odd prime number and $p \equiv 3 \pmod{4}$   & \cite[Theorem 7]{chung2010optimal} \\
$p$ & $L+1$ & $2g-1$ & \scell{$p=2Lg+1$ is an odd prime number, \\ $p \equiv 3 \pmod{4}$, $g$ is odd and $3\le g \le \frac{L+3}{2}$} & \cite[Theorem 9]{chung2010optimal} \\
$p^2$ & $p$ & $p$ & $p$ is prime & \cite[Theorem 2]{kumar1988frequency} \\
$p^t-1$ & $p^k$ & $p^{t-k}-1 $ & $p$ is  prime and $1\le k \le t$  & \cite[Theorem 2]{lempel1974families} \\
$p^r $ & $\frac{p^r-1}{f}$ & $f$ & $p=ef+1$ is an odd prime number and $f$ is odd & \cite[Theorem 3.1]{liu2013new} \\
$p^r $ & $\frac{p^r-1}{f}+1$ & $f-1$ & $p=ef+1$ is an odd prime number & \cite[Theorem 3.2]{liu2013new} \\
$q-1$ & $e$ & $f$ & $q=ef+1$ is a prime power and $f$ is even & \cite[Theorem 4]{ding2008sets}\\
$q-1$ & $e+1$ & $f-1$ & $q=ef+1$ is a prime power & \cite[Theorem 5]{ding2008sets} \\
$\frac{q^r-1}{l}$ & $q$ & $\frac{q^{r-1}-1}{l}$ & $q$ is a prime power, $l\mid (q-1)$ and $\gcd(l,\frac{q^r-1}{q-1})=1$ & \cite[Theorem 5]{ding2008sets} \\
$n$ & $\frac{n-1}{e}+1$ & $e-1$ & \scell{ $n=\prod_{i=1}^{k}{p_i^{e_i}}$, $p_i$ is odd and prime, \\ $e\mid p_i-1$, for $1\le i \le k$. $k>1$ or $k=1$ with $\frac{n-1}{e}\ge e$ } & \cite[Corollary 2]{zeng2013optimal}\\
$n$ & $\frac{n-1}{e}$ & $e$ & \scell{ $n=\prod_{i=1}^{k}{p_i^{e_i}}$, $p_i$ are odd and prime, \\ $e$ is odd and $e\mid (p_i-1)$, for $1\le i \le k$. }  & \cite[Theorem 5.5]{xu2018optimal}\\
$ev$ & $\frac{ev-1}{e-1}+1$ & $e-2$ & \scell{ $v=\prod_{i=1}^{k}{p_i^{e_i}}$, $p_i$ are odd and prime, \\ $e\ge 3$ and $e(e-1)\mid (p_i-1)$, $1\le i \le k$ } & \cite[Theorem 3]{zeng2013optimal} \\
$ev$ & $\frac{(e-1)v-1}{e-2}+1$ & $e-2$ & \scell{ $v=\prod_{i=1}^{k}{p_i^{e_i}}$, $p_i$ are odd and prime, \\ $e\ge 3$, $e\mid (p_i-1)$ and $(e-2)\mid (p_i-1)$, for $1\le i \le k$ } & \cite[Theorem 5.6]{xu2018optimal} \\
$ev$ & $\frac{(e-1)v-1}{e-2}$ & $e-2$ & \scell{ $v=\prod_{i=1}^{k}{p_i^{e_i}}$, $p_i$ are odd and prime, $e$ is odd,\\ $e\ge 3$, $e\mid (p_i-1)$ and $(e-2)\mid (p_i-1)$, for $1\le i \le k$ } & \cite[Theorem 5.7]{xu2018optimal} \\
$ev$ & $\frac{ev-1}{e-1}$ & $e-2$ & \scell{ $v=\prod_{i=1}^{k}{p_i^{e_i}}$, $p_i$ are odd and prime, $e$ is even \\ and $e(e-1)\mid (p_i-1)$, $1\le i \le k$ } & Corollary~\ref{con:app_fhs2}  \\
\hline
\end{tabular*}
\end{sidewaystable}

\begin{conc}\cite[Theorem 5.5]{xu2018optimal}\label{con:app_fhs1}
Let $n$ be an positive integer with the following factorization
$$n=\prod_{i=1}^{r}{p_i^{e_i}},$$
where $p_i$ are prime numbers, $e_i$ are positive integers $(i=1,2,\ldots,r)$. Let $e\ge 2$ be an odd integer such that $e \mid p_i-1$ for every $i$. Then there exist optimal $(n, \frac{n-1}{e}, e-1)$ FHSs over $(\Z_{n},+)$.
\end{conc}

\begin{conc}\label{con:app_fhs2}
Let $n$ be an positive integer with the following factorization
$$n=\prod_{i=1}^{r}{p_i^{e_i}},$$
where $p_i$ are prime numbers, $e_i$ are positive integers $(i=1,2,\ldots,r)$. Let $e\ge 2$ be an even integer such that $e(e-1) \mid p_i-1$ for every $i$. Then there exist optimal $(en, \frac{en-1}{e-1}, e-2)$ FHSs over $(\Z_{en},+)$.
\end{conc}
\begin{exam}
Let $n=13$ and $e=4$. Then an optimal $(52 , 17 , 3)$ FHS $\mathcal{T}_1$ is obtained from an $(52, 17, \{2, 3\})$ ZD function by Corollary~\ref{con:app_fhs2}, namely, \begin{align*}\begin{autobreak}\mathcal{T}_1=\{
1, 0, 0, 0, 1, 6, 7, 8, 4, 9, 10, 11, 1, 10, 11, 12, 2, 13, 14, 15, 4, 5, 6, 7, 4, 14, 15, 16, 3, 16, 13, 14, 3, 7, 8,
5, 1, 15, 16, 13, 2, 12, 9, 10, 3, 11, 12, 9, 2, 8, 5, 6
\}.\end{autobreak}\end{align*}
\end{exam}

To end this subsection, we summarize some known optimal FHSs in Table~\ref{tb:known_optimal_fhs}.

\section{Conclusion}\label{se:con}
In this paper, we have proposed a concept called zero-difference (ZD) as a generalization of zero-difference balanced (ZDB) by eliminating the ``balanced" requirement on ``zero-difference". Then we discussed some properties of ZD functions and construct ZD functions of Type-B from ZDB functions of Type-A by change point technic. Finally we show that these ZD functions can be used to construction optimal objects such as optimal CWCs, optimal DSSs and FHSs.

Using ZDB functions, many authors gave the conditions of the constructed CWCs, DSSs and FHSs being optimal. It is interesting that by our change point technic, the ZD functions obtained by those ZDB function can construct optimal CWCs, DSSs and FHSs without any conditions. The main reason is that these ZD functions are Type-B and hence almost balanced while those ZDB functions are Type-A and not balanced enough. Moreover, according to the ``philosophy" proposed by~\citeauthor{buratti2019partitioned}~\cite{buratti2019partitioned}, our ZD functions have no blocks of size 1 or 2 and possess relatively small difference between the smallest and largest block sizes. Thus these ZD functions can be viewed as the suboptimal function if no balanced ZDB functions with the same parameter exists.

In the future work, we are expected to construct more optimal CWC, DSS and FHS. For further applications in these areas, we have to investigate other requirements for these objects. They would impose other restrictions on the underlying functions. For example, there are bounds for a set of FHSs while only one bound of a FHS is considered in the paper.

\section*{Acknowledgment}
The research of Zongxiang~Yi was supported by the Talent Special Project of Research Project of Guangdong Polytechnic Normal University [Grant No. 2021SDKYA051] and Guangdong Basic and Applied Basic Research Fundation [Grant No. 2021A1515011954]. Tang's research was supported in part by the Foundation of National Natural Science of China [No. 61772147], Guangdong Province Natural Science Foundation of major basic research and Cultivation project [No. 2015A030308016], Project of Ordinary University Innovation Team Construction of Guangdong Province [No. 2015KCXTD014], Collaborative Innovation Major Projects of Bureau of Education of Guangzhou City [No. 1201610005] and National Cryptography Development Fund [No. MMJJ20170117].

\bibliographystyle{plainnat}


\end{document}